\documentclass[12pt,reqno]{amsart}
\usepackage{amsmath,amssymb,amsfonts,amscd,latexsym,amsthm,mathrsfs,verbatim,comment}
\usepackage{graphicx}
\usepackage[unicode]{hyperref}
\usepackage{hypbmsec}
\newtheorem{theorem}{Theorem}%[section]
\renewcommand{\d}{{\mathrm d}}

\renewcommand{\pmod}[1]{\;(\operatorname{mod}#1)}
\begin{document}

\title[Hypergeometric integrals for zeta values]{Some hypergeometric integrals\\for linear forms in zeta values}

\author{Wadim Zudilin}
\address{Department of Mathematics, IMAPP, Radboud University, PO Box 9010, 6500~GL Nijmegen, Netherlands}
\email{w.zudilin@math.ru.nl}

\dedicatory{To Carlo Viola, whose creativity in and love of integrals for linear forms in zeta values\\are boundless, on the occasion of his 75th birthday}


\date{11 April 2018}

\maketitle
%==================================================

In the exposition below, $s$ and $D$ are positive integers such that $s\ge3D-1$,
while the parameter $n$ is assumed to be a positive \emph{even} integer. The notation
$$
\zeta(s,\alpha) = \sum_{n=0}^\infty \frac1{(n+\alpha)^s}
$$
is used for the Hurwitz zeta function, so that $\zeta(s)=\zeta(s,1)$, and
$d_n = \operatorname{lcm}(1,2,\dots,n)$.

In \cite{FSZ18} the following approximations are constructed: for any $j\in\{1,\dots,D\}$, take
$$
r_{n,j} = \sum_{m=1}^\infty R_n\bigg(m+\frac{j}{D}\bigg),
\qquad\text{where}\quad
R_n(t) = D^{3Dn} n!^{s+1-3D} \, \frac{ \prod_{l=0}^{3Dn} (t-n+l/D)}{ \prod_{l=0}^n (t+l)^{s+1}}.
$$
It is shown that%
\footnote{Choosing $n$ even implies $3Dn+1+(s+1)(n+1)\equiv s\pmod2$, hence $R_n(-n-t)=(-1)^sR_n(t)$.
This reflects on the parity in summation in \eqref{forms}\,---\,consideration in \cite{FSZ18} is restricted to the case of $s$ odd.}
\begin{equation}
\label{forms}
r_{n,j} 
=a_{0,j}+\sum_{\substack{2\le i \le s\\i\equiv s\pmod2}}a_i\zeta\bigg(i,\frac{j}{D}\bigg),
\end{equation}
with
\begin{equation*}
\begin{gathered}
d_n^{s+1-i}a_i\in\mathbb Z \qquad\text{for}\quad i = 2,3,4,\dots,s, \quad i\equiv s\pmod2,
\\
d_{n+1}^{s+1}a_{0,j}\in\mathbb Z \qquad\text{for}\quad j\in \{1,\dots,D\}
\end{gathered}
\end{equation*}
(see \cite[Lemmas 1 and 2]{FSZ18}), and some further information is provided for the asymptotic growth of \emph{positive} quantities $r_{n,j}$ as $n\to\infty$. The approximations are building blocks for linear forms in zeta values $\zeta(i)$ with $i$ of the same parity as $s$, with the help of elementary formula
$$
\sum_{j=1}^d\zeta\biggl(i,\frac{j\,(D/d)}{D}\biggr)=\sum_{j=1}^d\zeta\biggl(i,\frac jd\biggr)=d^i\zeta(i)
$$
valid for any divisor $d$ of $D$.

The principal goal of this note is to establish the following integral representation of the approximations $r_{n,j}$ for $j\in\{1,\dots,D\}$.

\begin{theorem}
\label{th1}
The linear forms \eqref{forms} admit the integral representation
$$
r_{n,j}=\frac{D^{s-1}(3Dn+1)!}{n!^{3D}}\sum_{m=1}^D\xi^{-mj}r_{n,m}^*,
$$
where
$$
r_{n,m}^*
=\xi^m\idotsint\limits_{[0,1]^{s+1}}
\frac{\prod_{i=0}^sx_i^{Dn}(1-x_i^D)^n\,\d x_i}{(1-\xi^mx_0\dotsb x_s)^{3Dn+2}}
=\int_0^{\xi^m}\!\!\idotsint\limits_{[0,1]^s}
\frac{\prod_{i=0}^sx_i^{Dn}(1-x_i^D)^n\,\d x_i}{(1-x_0\dotsb x_s)^{3Dn+2}}
$$
and $\xi=\xi_D$ denotes a primitive root of unity of degree $D$.
\end{theorem}

\begin{proof}
As the rational function $R_n(t)$ has zeros at $t=m-(D-j)/D$ for $m=1,\dots,n$ and $j\in\{1,\dots,D\}$,
we can write
\begin{align}
r_{n,j}
&= \sum_{m=n}^\infty R_n\bigg(m+\frac{j}{D}\bigg)
= D^{3Dn} n!^{s+1-3D}\sum_{k=0}^\infty
\frac{ \prod_{l=0}^{3Dn} (k+(l+j)/D)}{ \prod_{l=0}^n (k+n+l+j/D)^{s+1}}
\nonumber \displaybreak[2]\\
&=\frac{n!^{s+1-3D}\prod_{l=0}^{3Dn} (l+j)}{D\prod_{l=0}^n (n+l+j/D)^{s+1}}
\nonumber \\ &\quad\times
{}_{s+D+1}F_{s+D}\biggl(\begin{matrix} \{3n+\frac{j+l}D:l=1,\dots,D\}, \, \{n+\frac jD\}^{s+1} \\[2.5pt]
\{1+\frac{j-l}D:l=1,\dots,D,\,j\ne l\}, \, \{2n+1+\frac jD\}^{s+1} \end{matrix}\biggm|1\biggr)
\nonumber \displaybreak[2]\\
&=\frac{(3Dn+j)!}{D\,n!^{3D}(j-1)!}
\idotsint\limits_{[0,1]^{s+1}}f_j(t_0\dotsb t_s)
\prod_{i=0}^st_i^{n+j/D-1}(1-t_i)^n\,\d t_i,
\label{forms2}
\end{align}
where
\begin{align*}
f_j(t)
&={}_DF_{D-1}\biggl(\begin{matrix} \{3n+\frac{j+l}D:l=1,\dots,D\} \\
\{1+\frac{j-l}D:l=1,\dots,D,\,j\ne l\} \end{matrix} \biggm| t \biggr)
\\
&=\sum_{k=0}^\infty\frac{\prod_{l=1}^D(3n+\frac{j+l}D)_k}{\prod_{l=1}^D(1+\frac{j-l}D)_k}\,t^k
=\sum_{k=0}^\infty\frac{(3Dn+j+1)_{Dk}}{(j)_{Dk}}\,t^k
\qquad\text{for}\quad j\in\{1,\dots,D\}.
\end{align*}
Using
$$
\sum_{l=0}^\infty\frac{(a)_l}{l!}\,x^l=\frac1{(1-x)^a}
$$
observe that
\begin{align*}
\frac{(3Dn+2)_{j-1}}{(j-1)!}\,x^{j-1}f_j(x^D)
&=\sum_{k=0}^\infty\frac{(3Dn+2)_{Dk+j-1}}{(Dk+j-1)!}\,x^{Dk+j-1}
\\
&=\sum_{\substack{l=0\\l\equiv j-1\pmod D}}^\infty\frac{(3Dn+2)_l}{l!}\,x^l
=\frac1D\sum_{m=1}^D\frac{\xi^{-m(j-1)}}{(1-\xi^mx)^{3Dn+2}}.
\end{align*}
Taking $t_i=x_i^D$ for $i=0,1,\dots,s$ in the integrals \eqref{forms2} we thus obtain
\begin{align*}
r_{n,j}
&=\frac{D^{s-1}(3Dn+1)!}{n!^{3D}}
\sum_{m=1}^D\xi^{-m(j-1)}\idotsint\limits_{[0,1]^{s+1}}
\frac{\prod_{i=0}^sx_i^{Dn}(1-x_i^D)^n\,\d x_i}{(1-\xi^mx_0\dotsb x_s)^{3Dn+2}}
\end{align*}
for each $j\in\{1,\dots,D\}$.
\end{proof}

Taking $D=2$ and $s\ge5$ odd, we obtain the linear forms
\begin{align*}
7r_{n,2}-r_{n,1}
&=\frac{2^s(6n+1)!}{n!^6}\idotsint\limits_{[0,1]^{s+1}}
\biggl(\frac3{(1-x_0x_1\dotsb x_s)^{6n+2}}
\\[-10.5pt] &\qquad\qquad\qquad\qquad\qquad\qquad
-\frac4{(1+x_0x_1\dotsb x_s)^{6n+2}}\biggr)
\prod_{i=0}^sx_i^{2n}(1-x_i^2)^n\,\d x_i
\\
&=\frac{2^s(6n+1)!}{n!^6}\idotsint\limits_{\gamma\times[0,1]^s}
\frac{\prod_{i=0}^sx_i^{2n}(1-x_i^2)^n\,\d x_i}{(1-x_0x_1\dotsb x_s)^{6n+2}}
\end{align*}
in $\mathbb Q+\mathbb Q\zeta(5)+\dots+\mathbb Q\zeta(s)$ considered previously in \cite{Zu18}.
Here the path $\gamma\subset\mathbb R$ for integrating with respect to $x_0$ is given by
$\gamma=3[0,1]+4[0,-1]$, and the parity assumption on $n$ can be dropped.

\subsection*{Acknowledgements}
The note was produced during the trimester on \emph{Periods in Number Theory, Algebraic Geometry and Physics}
at the Hausdorff Research Institute for Mathematics (Bonn, Germany).
I thank Cl\'ement Dupont for his encouragement to write the integrals for the hypergeometric approximations used in \cite{FSZ18,Zu18}.

%==================================================


\begin{thebibliography}{9}

\bibitem{FSZ18}
\textsc{S. Fischler}, \textsc{J. Sprang} and \textsc{W. Zudilin},
Many odd zeta values are irrational,
\emph{Preprint} \texttt{arXiv:\,1803.08905 [math.NT]} (2018).

\bibitem{Zu18}
\textsc{W. Zudilin},
One of the odd zeta values from $\zeta(5)$ to $\zeta(25)$ is irrational. By elementary means,
\emph{SIGMA} \textbf{14} (2018),  no.~028, 8 pages;
\emph{Preprint} \texttt{arXiv:\,1801.09895 [math.NT]} (2018).

\end{thebibliography}
\end{document}